\documentclass[11pt, a4paper]{article}

\usepackage{amsmath}
\usepackage{amssymb}
\usepackage{amsthm}
\usepackage{appendix}
\usepackage{verbatim}

\newtheorem{theorem}{Theorem}[section]

\newtheorem{lemma}{Lemma}[section]
\newtheorem{remark}{Remark}[section]

\begin{document}
\title{Hardy Subspaces with Sparse Fourier Spectrum and M\"{u}ntz Spaces}
\author{Elias Zikkos\\
Department of Mathematics, Khalifa University, \\
Abu Dhabi, United Arab Emirates\\
email address:  elias.zikkos@ku.ac.ae and eliaszikkos@yahoo.com}

\maketitle

\begin{abstract}
Let $\Lambda=\{\lambda_n\}_{n=1}^{\infty}\subset\mathbb{N}$ with $\lambda_n$ strictly increasing and such that
$\sum_{n=1}^{\infty}\lambda_n^{-1}<\infty$.
We show that a Hardy subspace $H^2 (\mathbb{D}, \Lambda)$ consisting of functions with sparse Fourier spectrum $\Lambda$
coincides with a M\"{u}ntz space $\overline{M^2_{\Lambda}}(\mathbb{D})$ characterized by square-summability of coefficients relative to a biorthogonal family. As consequences, we obtain a new characterization of the Hardy norm in $H^2 (\mathbb{D}, \Lambda)$ and an integral representation formula for the Fourier coefficients. The proof uses the biorthogonal representation developed in
\cite{Zikkos2026Acta}.
\end{abstract}



\section{Introduction and our main result}

Let $\mathbb{D}:=\{z:\, |z|<1\}$ be the unit disk and let  $H^2(\mathbb{D})$ be the Hardy space of functions analytic in
$\mathbb{D}$ such that
\[
\sup_{0<r<1} \int_{0}^{2\pi} |f(re^{i\theta})|^2\, d\theta<\infty.
\]
It is well known that an analytic function $f(z)=\sum_{n=0}^{\infty} a_n z^n$ in
$\mathbb{D}$ belongs to the space $H^2(\mathbb{D})$
if and only if $\sum_{n=0}^{\infty} |a_n|^2<\infty$. The standard norm is
\[
||f||_{H^2} : = \left(\sum_{n=0}^{\infty} |a_n|^2\right)^{1/2}.
\]
The Taylor coefficients $a_n$ are equal to the Fourier coefficients
\[
c_n=\frac{1}{2\pi}\int_{0}^{2\pi} f(e^{i\theta})\cdot e^{-i n\theta}\, d\theta,\qquad n=0, 1, 2, \dots
\]
of its boundary function $f(e^{i\theta})$, and the Fourier coefficients $c_n$ vanish for negative $n$ (\cite[Theorem 3.4]{Duren}).

For various properties of functions in $H^p(\mathbb{D})$ spaces one may consult \cite{Duren,Jevtic,Koosis,Mashreghi}.

\begin{remark}
Consider a sequence $\Lambda=\{\lambda_n\}_{n=1}^{\infty}$ which is a subset of the integer set $\mathbb{N}$,
with $\lambda_n$ strictly increasing, and such that
\begin{equation}\label{convergence}
\sum_{n=1}^{\infty} \frac{1}{\lambda_n}<\infty.
\end{equation}
The main result of this paper is Theorem \ref{HS}. We show that the Hardy subspace $H^2 (\mathbb{D}, \Lambda)$ with sparse Fourier spectrum $\Lambda$ is precisely the M\"{u}ntz space $\overline{M^2_{\Lambda}}(\mathbb{D})$ whose coefficient sequence
$\{\langle f, r_n\rangle_{L^2 (0,1)}\}$ is square-summable, where $\{r_n\}_{n=1}^{\infty}$ is a family
biorthogonal to the system $\{t^{\lambda_n}\}_{n=1}^{\infty}$ in $L^2(0,1)$.
This creates a connection between Hardy space theory and the geometry of M\"{u}ntz spaces, allowing techniques developed for M\"{u}ntz spaces
to be interpreted in terms of sparse Fourier expansions. In particular, the result provides a new characterization of the Hardy norm through biorthogonal coefficients as well as an integral representation formula for the Fourier coefficients.
\end{remark}

We let $H^2 (\mathbb{D}, \Lambda)$ consist of all the functions $f\in H^2 (\mathbb{D})$
such that their Fourier coefficients vanish whenever $n\notin \Lambda$.
That is,

\begin{align*}
H^2(\mathbb{D}, \Lambda): & = \left\{f\in H^2(\mathbb{D}):\,\, f(z)=\sum_{n=1}^{\infty} c_{\lambda_n} z^{\lambda_n}\right\}\\
& =  \left\{\text{f is analytic on $\mathbb{D}$}:\,\, f(z)=\sum_{n=1}^{\infty} c_{\lambda_n} z^{\lambda_n},\,\,
\{c_{\lambda_n}\}\in \ell^2\}\right\}.
\end{align*}

In order to introduce the M\"{u}ntz space $\overline{M^2_{\Lambda}}(\mathbb{D})$, we first
denote by $M_{\Lambda}$ the system $\{e_n(t): = t^{\lambda_n}\}_{n=1}^{\infty}$ and by $\overline{M_{\Lambda}}$ the closed span of
$M_{\Lambda}$ in $L^2 (0,1)$. Due to \eqref{convergence}, it follows from the M\"{u}ntz-Sz\'{a}sz theorem that $\overline{M_{\Lambda}}$
is a proper subspace of $L^2 (0,1)$. Moreover, as proved by Clarkson and Erd\H{o}s
(see \cite{CE} and \cite[Corollary 6.2.4]{Gurariy}),
any function $f\in\overline{M_{\Lambda}}$ extends analytically in $\mathbb{D}$ and it is represented as a power series of the form
\begin{equation}\label{analyticext}
f(z)=\sum_{n=1}^{\infty} a_{\lambda_n} z^{\lambda_n}.
\end{equation}
\begin{remark}
We denote the space of these analytic extensions by $\overline{M_{\Lambda}}(\mathbb{D})$.
\end{remark}

We point out that $M_{\Lambda}$ is a minimal system in $\overline{M_{\Lambda}}$, in other words,
every element of $M_{\Lambda}$ does not belong to the closed span of the remaining elements in $L^2 (0,1)$
(see \cite[Proposition 6.1.4]{Gurariy}). This minimal system has a unique biorthogonal family in $\overline{M_{\Lambda}}$ ,
that is there exists a system of functions
$r_{\Lambda}:=\{r_n(t)\}_{n=1}^{\infty}\subset \overline{M_{\Lambda}}$ such that
\[
\langle e_n, r_m\rangle_{L^2 (0,1)} =\begin{cases} 1, & m=n, \\  0, & m\not=n,\end{cases}\qquad \text{where}\qquad
\langle f, g\rangle_{L^2 (0,1)}:=\int_{0}^{1} f(t)\overline{g(t)}\, dt.
\]

In \cite[Theorem 1]{Zikkos2026Acta} we investigated the properties of the biorthogonal families $M_{\Lambda}$ and $r_{\Lambda}$
inside the space $\overline{M_{\Lambda}}$. Motivated by the Clarkson-Erd\H{o}s theorem, we proved that
for each $n\in\mathbb{N}$, the $a_{\lambda_n}$ coefficient appearing in \eqref{analyticext} is equal to the
inner product $\langle f, r_{n} \rangle_{L^2 (0,1)}$.
Moreover, we showed that the family  $r_{\Lambda}$ is complete in $\overline{M_{\Lambda}}$,
in other words the closed span of $r_{\Lambda}$ in $\overline{M_{\Lambda}}$ is equal to $\overline{M_{\Lambda}}$; hence the families
$r_{\Lambda}$ and $M_{\Lambda}$ are Markushevich bases for the space $\overline{M_{\Lambda}}$.
We even proved that these families are strong Markushevich bases and applied the results to constructing compact operators on $\overline{M_{\Lambda}}$ that admit spectral synthesis (see \cite[Theorems 2, 3, 4]{Zikkos2026Acta}).

A partial result from \cite[Theorem 1]{Zikkos2026Acta} is the following.
\begin{theorem}\label{biorthogonalsystem}
Let $\Lambda\subset\mathbb{N}$ such that $\Lambda$ satisfies $(1.1)$. Then there exists a family of functions
\[
r_{\Lambda}=\{r_{n}:\,\, n\in\mathbb{N}\}\subset \overline{M_{\Lambda}}
\]
so that $r_{\Lambda}$ is the unique biorthogonal sequence to $M_{\Lambda}$ in the space $\overline{M_{\Lambda}}$, and such that
each function $f\in \overline{M_{\Lambda}}$ extends as an analytic function in the unit disk $\mathbb{D}$,
hence its analytic extension belongs to $\overline{M_{\Lambda}}(\mathbb{D})$, so that
\begin{equation}\label{representationf}
f(z)=\sum_{n=1}^{\infty}\langle f, r_{n} \rangle_{L^2 (0,1)} z^{\lambda_n},
\end{equation}
converging uniformly on compact subsets of $\mathbb{D}$.
\end{theorem}

We now define the particular subspace of $\overline{M_{\Lambda}}(\mathbb{D})$ which coincides with the space
$H^2(\mathbb{D}, \Lambda)$.
Let
\[
\overline{M^2_{\Lambda}}(\mathbb{D}):=\left\{f\in \overline{M_{\Lambda}}(\mathbb{D}):\,\,
\{\langle f, r_n\rangle_{L^2 (0,1)}\}\in \ell^2\right\}.
\]
Then the following is true. The proof is given in Section 4.
\begin{theorem}\label{HS}
Let $\Lambda=\{\lambda_n\}_{n=1}^{\infty}\subset\mathbb{N}$ such that $\Lambda$
satisfies \eqref{convergence}.
Let $r_{\Lambda}$ be the biorthogonal family to $M_{\Lambda}$ as in Theorem \ref{biorthogonalsystem}.
Then
\begin{equation}\label{equivalence}
H^2(\mathbb{D}, \Lambda)=\overline{M^2_{\Lambda}}(\mathbb{D}).
\end{equation}
Moreover
\begin{equation}\label{norms}
||f||_{H^2}^2=\sum_{n=1}^{\infty}\left| \langle f, r_n\rangle_{L^2 (0,1)} \right|^2.
\end{equation}
In addition, the following holds for the Fourier $c_{\lambda_n}$ coefficients of $f\in H^2 (\mathbb{D}, \Lambda)$: for every fixed $\theta\in [0,2\pi]$
\begin{equation}\label{coeff}
c_{\lambda_n}=\left(\int_{0}^{1} f(te^{i\theta})\cdot \overline{r_n(t)}\, dt\right)e^{-i\theta\lambda_n}\qquad n=1,2,\dots
\end{equation}
\end{theorem}

\section{Some other results}
\setcounter{equation}{0}

Now, suppose that the coefficients $\{a_{\lambda_n}\}$ of $f$ as in \eqref{analyticext} belong to the $\ell^2$ space: then
clearly $f\in H^2 (\mathbb{D}, \Lambda)$. We examine whether the converse implication holds.

``If $f\in H^2 (\mathbb{D}, \Lambda)$, is it true that the restriction of $f$ on the interval $[0,1)$ belongs to $\overline{M_{\Lambda}}$?''

The answer is affirmative; in fact it is a special case of the following result.
\begin{theorem}\label{first}
Suppose that the Fourier coefficients of a function $f\in H^2 (\mathbb{D})$ vanish for all $n\notin A$ where $A\subset \mathbb{N}\cup\{0\}$.
Then the restriction of $f$ on $[0,1)$ belongs to the closed span of the family $\{t^n\}_{n\in A}$ in $L^2 (0,1)$.
\end{theorem}
\begin{proof}
It is a combination of Lemmas \ref{1} and \ref{converse}. In Lemma \ref{1} we prove that the restriction of $f$ on $[0,1)$ belongs to the $L^2 (0,1)$ space, whereas in Lemma \ref{converse} we show that if $f\in L^2 (0,1)$ and  $f(t)=\sum_{n\in A}a_n t^n$ on $[0,1)$ with uniform convergence on compact subsets of $[0,1)$, then $f$ belongs to the closed span of $\{t^n\}_{n\in A}$ in $L^2 (0,1)$.
\end{proof}

Next we ask whether the converse of Theorem \ref{first} is true.

``If a function $f\in L^2(0,1)$  belongs to the closed span of a family $\{t^n\}_{n\in A}$ in $L^2 (0,1)$ where
$A\subset \mathbb{N}\cup\{0\}$, is it true that $f$ extends analytically in $\mathbb{D}$ and its extension
belongs to the space $H^2 (\mathbb{D})$, with $f$ admitting the representation $f(z)=\sum_{n\in A} a_n z^{n}$ and $\{a_n\}\in \ell^2$?''\\

\smallskip

Clearly the answer is negative when $\sum_{n\in A}1/n=\infty$ due to the M\"{u}ntz-Sz\'{a}sz theorem.
If such a series diverges then the closed span of $\{t^n\}_{n\in A}$ is equal to $L^2(0,1)$.

But what about if $\sum_{n\in A}1/n<\infty$? The answer is negative in general, as the following result demonstrates whose proof
is given in Section 3.
\begin{theorem}\label{lacunary}
Let $\Lambda=\{\lambda_n\}_{n=1}^{\infty}$ be a lacunary sequence of natural numbers, that is,
there exists some $q>1$ so that $\lambda_{n+1}/\lambda_n>q$ for all $n\in\mathbb{N}$. Then
\[
g(z):=\sum_{n=1}^{\infty} z^{\lambda_n}
\]
is analytic in $\mathbb{D}$ but clearly $g\notin H^2 (\mathbb{D})$. However,
$g(t)$ belongs to $L^2 (0,1)$ and moreover $g(t)$ belongs to the closed span of $\{t^{\lambda_n}\}_{n=1}^{\infty}$ in $L^2 (0,1)$.
Hence, by definition, $g$ belongs to the space $\overline{M_{\Lambda}}(\mathbb{D})$.
\end{theorem}

Therefore, even in the case when $\Lambda\subset\mathbb{N}$ satisfying $(1.1)$, the condition that a function $f$ belongs to the space
$\overline{M_{\Lambda}}$ does not guarantee alone that the analytic extension of $f$ in $\mathbb{D}$ belongs to $H^2 (\mathbb{D})$.
In other words, the space $\overline{M_{\Lambda}}(\mathbb{D})$ is not a subspace of $H^2 (\mathbb{D}, \Lambda)$.
\begin{remark}
This clearly changes if $f\in \overline{M^2_{\Lambda}}(\mathbb{D})$.
\end{remark}

\section{Auxiliary Lemmas and Theorem \ref{lacunary}}
\setcounter{equation}{0}

\subsection{The radial integrals $\int_0^1 |f(te^{i\theta})|^2\, dt$ are bounded}

Let
\[
f(z)=\sum_{n=0}^{\infty} c_n z^n
\]
be a function in the Hardy space $H^2(\mathbb{D})$.
We study the radial integrals
\[
\int_0^1 |f(te^{i\theta})|^2\, dt
\]
for a fixed $\theta\in[0,2\pi]$.

Although radial limits of $H^2$ functions need not exist for every $\theta$, the result below shows that the radial function
\[
t \mapsto f(te^{i\theta})
\]
belongs to $L^2(0,1)$ for each $\theta\in [0, 2\pi]$ and its norm is uniformly controlled by the $H^2$ norm of $f$.

\begin{lemma}\label{1}
If $f\in H^2(\mathbb{D})$, then for every $\theta\in[0,2\pi]$,
\begin{equation}\label{upperbound}
\int_{0}^{1} |f(te^{i\theta})|^2\, dt \le \pi ||f||_{H^2}^2.
\end{equation}
\end{lemma}

\begin{proof}
We have $f(z)=\sum_{n=0}^{\infty}c_n z^n$ with $\{c_n\}\in\ell^2$.
Fix $\theta$ and write
\[
f(te^{i\theta}) = \sum_{n=0}^{\infty} c_ne^{i\theta n} t^n.
\]

Hence for any $\rho\in (0,1)$ we have
\begin{align*}
\int_{0}^{\rho} |f(te^{i\theta})|^2\, dt & =
\int_{0}^{\rho} \left|\sum_{n=0}^{\infty} c_ne^{i\theta n} t^n\right|^2 dt\\
& =  \int_{0}^{\rho} \sum_{n=0}^{\infty} \sum_{k=0}^{\infty}c_n\overline{c_k}e^{i\theta (n-k)}t^{n+k}\, dt\\
& \le  \sum_{n=0}^{\infty} \sum_{k=0}^{\infty}\frac{|c_n||c_k|}{n+k+1}.
\end{align*}

Since $\{c_n\}\in\ell^2$ then the double series
\[
\sum_{n=0}^{\infty} \sum_{k=0}^{\infty}\frac{|c_n||c_k|}{n+k+1}
\]
converges as a special case of Hilbert's Inequality \cite[Theorem 323]{Polya} and
in fact we have
\[
\sum_{n=0}^{\infty} \sum_{k=0}^{\infty}\frac{|c_n||c_k|}{n+k+1}\le \pi \sum_{n=0}^{\infty}|c_n|^2=\pi ||f||^2_{H^2}.
\]

Combining the above shows that for any $\rho\in (0,1)$ we have
\[
\int_{0}^{\rho} |f(te^{i\theta})|^2\, dt \le \pi ||f||_{H^2}^2.
\]
This uniform upper bound implies that \eqref{upperbound} is true.
\end{proof}

\subsection{A converse result to Clarkson-Erd\H{o}s}

The following result is known  in the case when $\Lambda=\{\lambda_n\}_{n=1}^{\infty}$ is a sequence of
positive real numbers in an increasing order such that \eqref{convergence} holds (see \cite[Corollary 6.2.4]{Gurariy} and \cite[Lemma 1]{Zikkos2026Acta}).  We reprove the result as given by us
in \cite[Lemma 1]{Zikkos2026Acta} $\bf without$ taking into assumption the convergence of the series in \eqref{convergence}.
In other words, the result holds even when the series $\sum_{n=1}^{\infty} 1/\lambda_n=\infty$. This $extension$ is needed for proving
Theorem \ref{first}.

\begin{lemma}\label{converse}
Let $\Lambda=\{\lambda_n\}_{n=1}^{\infty}\subset\mathbb{N}$ with the $\lambda_n$ in a strictly increasing order.
Suppose that $f\in L^2 (0,1)$ and $f(t)=\sum_{n=1}^{\infty} c_n t^{\lambda_n}$ for $t\in [0,1)$ with the series converging uniformly on compact
subsets of $[0, 1)$. Then $f$ belongs to the closed span of the family $\{t^{\lambda_n}\}_{n=1}^{\infty}$ in $L^2 (0,1)$.
\end{lemma}

\begin{proof}

Fatou's Lemma and changing variables gives

\begin{align*}
\int_{0}^{1}|f(t)|^2\,dt\le \liminf_{\rho\to 1^-}\int_{0}^{1}|f(\rho t)|^2\,dt & \le
\limsup_{\rho\to 1^-}\int_{0}^{1}|f(\rho t)|^2\,dt \\
 & =  \limsup_{\rho\to 1^-}\int_{0}^{\rho}|f(u)|^2\, \frac{du}{\rho}\\ & \le
\limsup_{\rho\to 1^-}\frac{1}{\rho}\cdot \int_{0}^{1}|f(u)|^2\, du \\ & =
\int_{0}^{1}|f(t)|^2\,dt.
\end{align*}

The above implies that $\lim_{\rho\to 1^-}\int_{0}^{1}|f(\rho t)|^2\,dt$ exists and

\begin{equation}\label{limit}
\lim_{\rho\to 1^-}\int_{0}^{1}|f(\rho t)|^2\, dt = \int_{0}^{1}|f(t)|^2\, dt.
\end{equation}
Obviously one has
\[
|f(\rho t)-f(t)|^2\to 0\,\,\text{as}\,\, \rho\to 1^-
\]
and
\[
|f(\rho t)-f(t)|^2\le 2(|f(\rho t)|^2 + |f(t)|^2).
\]
Clearly
\[
2(|f(\rho t)|^2 + |f(t)|^2)\to 4|f(t)|^2\,\,\text{as}\,\, \rho\to 1^-,
\]
and from \eqref{limit} we get
\[
\int_{0}^{1} 2(|f(\rho t)|^2 + |f(t)|^2)\, dt \to \int_{0}^{1} 4|f(t)|^2\, dt \,\,\text{as}\,\, \rho\to 1^-.
\]
It then follows from the
\[
\text{\it{Generalized Lebesgue Convergence Theorem}}
\]
that
\[
\lim_{\rho\to 1^-}\int_{0}^{1}|f(\rho t)-f(t)|^2\, dt = 0.
\]
Therefore, for every $\epsilon>0$ there is some $0<\delta_{\epsilon}<1$ so that for all $\rho\in (\delta_{\epsilon}, 1)$ one has
\begin{equation}\label{epsilon}
\int_{0}^{1}|f(\rho t)-f(t)|^2\, dt < \epsilon.
\end{equation}

Next, for fixed $\epsilon$, $\delta_{\epsilon}$, as well as $\rho \in (\delta_{\epsilon}, 1)$, the series

\[
f(\rho t) = \sum_{n=1}^{\infty} c_{n} \cdot \rho^{\lambda_n}\cdot t^{\lambda_n}
\]
converges uniformly on the interval $[0,1]$. Hence, there is some positive integer $N$, depending on $\epsilon$ and $\rho$, such that
\[
\int_{0}^{1}\left| f(\rho t) - \sum_{n=1}^{N} c_{n} \cdot \rho^{\lambda_n}\cdot t^{\lambda_n}\right|^2\, dt<\epsilon.
\]
Combining this with \eqref{epsilon} and applying Minkowski's inequality shows that
\[
\int_{0}^{1}\left| f(t) - \sum_{n=1}^{N} c_{n} \cdot \rho^{\lambda_n}\cdot t^{\lambda_n}\right|^2\, dt<2\epsilon.
\]
Clearly now we conclude that $f$ belongs to the space $\overline{M_{\Lambda}}$.
\end{proof}

\subsection{Proof of Theorem \ref{lacunary}}

We have
\[
|g(t)|^2=\left(\sum_{n=1}^{\infty} t^{\lambda_n}\right)\left(\sum_{k=1}^{\infty} t^{\lambda_k}\right)=
\sum_{n=1}^{\infty}\sum_{k=1}^{\infty} t^{\lambda_n + \lambda_k}.
\]
Since $\int_{0}^{1} t^{\lambda_n + \lambda_k}\, dt<1/(\lambda_n + \lambda_k)$ we need to evaluate the double series
\[
S=\sum_{n=1}^{\infty}\sum_{k=1}^{\infty}\frac{1}{\lambda_n + \lambda_k}.
\]
If $S$ is a positive finite number then for any $\rho\in (0,1)$ we have
\[
\int_{0}^{\rho}|g(t)|^2=\sum_{n=1}^{\infty}\sum_{k=1}^{\infty}\int_{0}^{\rho}  t^{\lambda_n + \lambda_k} \, dt \le
\sum_{n=1}^{\infty}\sum_{k=1}^{\infty} \frac{1}{\lambda_n + \lambda_k}=S<\infty
\]
and this shows that $g\in L^2 (0,1)$.

To evaluate $S$, split the double sum into the diagonal terms $m=n$ and the off--diagonal terms: due to symmetry, we have

\[
S=\sum_{n=1}^{\infty}\frac{1}{\lambda_n + \lambda_n}
+2\sum_{n=1}^{\infty}\sum_{k=n+1}^{\infty}\frac{1}{\lambda_n + \lambda_k}.
\]

We easily get

\[
\sum_{n=1}^{\infty}\frac{1}{\lambda_n + \lambda_n}
=\sum_{n=1}^{\infty}\frac{1}{2\lambda_n}<\infty.
\]

Regarding $\sum_{n=1}^{\infty}\sum_{k=n+1}^{\infty}\frac{1}{\lambda_n + \lambda_k}$, for each fixed $n$ and $k\ge n+1$ we write $k=n+j$ for $j=1,2,\dots$. Since $\{\lambda_{n}\}$ is lacunary with $\lambda_{n+1}/\lambda_n>q>1$ then $\lambda_{n+j}>q^j\lambda_n$.
Thus $\lambda_n+\lambda_{n+j}>q^j\lambda_n$. Hence

\[
\sum_{n=1}^{\infty}\sum_{k=n+1}^{\infty}\frac{1}{\lambda_n + \lambda_k}
<
\sum_{n=1}^{\infty}\sum_{j=1}^{\infty}
\frac{1}{\lambda_n}\cdot q^{-j}=\left(\sum_{n=1}^{\infty}\frac{1}{\lambda_n}\right)\left(\sum_{j=1}^{\infty}q^{-j}\right)<\infty.
\]

Combining the above shows that $S$ is a positive finite number thus $g\in L^2 (0,1)$.
It then follows from Lemma \ref{converse} that $g$ belongs to the closed span of $\{t^{\lambda_n}\}_{n=1}^{\infty}$ in $L^2 (0,1)$.

\section{Proof of Theorem \ref{HS}}
\setcounter{equation}{0}

Suppose first that $f(z)\in \overline{M^2_{\Lambda}}(\mathbb{D})$ thus by definition
\[
f(z)=\sum_{n=1}^{\infty} \langle f, r_n\rangle_{L^2 (0,1)} z^{\lambda_n}
\]
and
\[
\{\langle f, r_n\rangle_{L^2 (0,1)}\} \in\ell^2.
\]
Then obviously $f\in H^2 (\mathbb{D}, \Lambda)$.

Suppose now that $f\in H^2 (\mathbb{D}, \Lambda)$ thus $f(z)=\sum_{n=1}^{\infty} c_{\lambda_n} z^{\lambda_n}$ and $\{c_{\lambda_n}\}\in\ell^2$.
From Lemma \ref{1} we know that $f(t)$ belongs to the space $L^2 (0,1)$.
It then follows from Lemma \ref{converse} that $f(t)\in\overline{M_{\Lambda}}$, thus from \eqref{representationf}
we have the series representation $f(z)=\sum_{n=1}^{\infty}\langle f, r_n\rangle_{L^2 (0,1)} z^{\lambda_n}$.
But the coefficients of a power series are unique,
thus for each $n\in\mathbb{N}$ we have $c_{\lambda_n}=\langle f, r_n\rangle_{L^2 (0,1)}$. Since $\{c_{\lambda_n}\}\in\ell^2$ then
$\{\langle f, r_n\rangle_{L^2 (0,1)}\}\in\ell^2$. Therefore we conclude that $f(z)\in \overline{M^2_{\Lambda}}(\mathbb{D})$.

Hence \eqref{equivalence} holds, we have $H^2(\mathbb{D}, \Lambda)=\overline{M^2_{\Lambda}}(\mathbb{D})$.\\

\smallskip

Next, if $f\in H^2 (\mathbb{D}, \Lambda)$ then $||f||_{H^2}^2=\sum_{n=1}^{\infty}|c_{\lambda_n}|^2$. Since
$c_{\lambda_n}=\langle f, r_n\rangle_{L^2 (0,1)}$ then relation \eqref{norms} is true. \\

\smallskip

Finally we deal with relation \eqref{coeff}. Let $f$ belong to $H^2 (\mathbb{D}, \Lambda)$
thus $f(z)=\sum_{n=1}^{\infty} c_{\lambda_n} z^{\lambda_n}$.
Fix some $\theta\in [0,2\pi)$ so $f(te^{i\theta})=\sum_{n=1}^{\infty} c_{\lambda_n}\cdot e^{i\lambda_n\theta}\cdot t^{\lambda_n}$.
By Lemma \ref{1} we know that $f(te^{i\theta})$ belongs to the space $L^2 (0,1)$. It follows from Lemma \ref{converse}
that $f(te^{i\theta})$ belongs to $\overline{M_{\Lambda}}$, thus from \eqref{representationf} we have
\[
f(te^{i\theta})=\sum_{n=1}^{\infty}\left(\int_{0}^{1} f(te^{i\theta})\cdot \overline{r_n(t)}\, dt\right)\cdot t^{\lambda_n},
\]
with the series converging uniformly on compact subsets of $[0,1)$. By the uniqueness of coefficients of power series we conclude that
\[
c_{\lambda_n}\cdot e^{i\lambda_n\theta}=\int_{0}^{1} f(te^{i\theta})\cdot \overline{r_n(t)}\, dt,\qquad n=1,2,\dots.
\]
This proves \eqref{coeff}.

The proof of Theorem \ref{HS} is now complete.


\begin{thebibliography}{99}
\bibitem{CE} J. A. Clarkson, P. Erd\H{o}s,
\textit{Approximation by polynomials},
Duke Math. J. \textbf{10}, 5--11 (1943).

\bibitem{Duren} P. L. Duren,
\textit{Theory of $H^p$ spaces},
Pure and Applied Mathematics, $\bf 38$ (Academic Press, New York), p. xii+258, 1970.

\bibitem{Gurariy} V. I. Gurariy, W. Lusky,
\textit{Geometry of M\"{u}ntz spaces and related questions},
Lecture Notes in Mathematics, 1870.
Springer-Verlag, Berlin, xiv+172 pp.
ISBN: 978-3-540-28800-8; 3-540-28800-7, 2005.

\bibitem{Polya} G. H. Hardy, J. E. Littlewood, G. P\'{o}lya,
\textit{Inequalities},
2nd ed., Cambridge Univ. Press, Cambridge, 1952.

\bibitem{Jevtic} M. Jevti\'{c}, D. Vukoti\'{c}, M. Arsenovi\'{c},
\textit{Taylor coefficients and coefficient multipliers of Hardy and Bergman-type spaces},
Springer, Electronic ISBN: 978-3-319-45644-7, 2016.

\bibitem{Koosis} P. Koosis,
\textit{Introduction to Hp spaces},
Cambridge University Press,
Online ISBN: 9780511470950, 2009.

\bibitem{Mashreghi}  J. Mashreghi,
\textit{Representation theorems in Hardy spaces},
Cambridge University Press, Cambridge,
ISBN: 978-0-521-73201-7, 2009.

\bibitem{Zikkos2026Acta} E. Zikkos,
\textit{M\"{u}ntz spaces: Strong Markushevich bases and Spectral Synthesis},
Acta Sci. Math. (Szeged), https://doi.org/10.1007/s44146-026-00235-8, (2026).

\end{thebibliography}

\end{document}